\documentclass{birkjour}
 \newtheorem{thm}{Theorem}[section]
 \newtheorem{cor}[thm]{Corollary}
 \newtheorem{lem}[thm]{Lemma}
 
 \theoremstyle{definition}
 \newtheorem{defn}[thm]{Definition}
 \theoremstyle{remark}
 \newtheorem{rem}[thm]{Remark}
 \newtheorem*{ex}{Example}
 \numberwithin{equation}{section}
 \newtheorem{alg}[thm]{Algorithm}

\usepackage{amssymb, amsthm, amsmath,color}
\usepackage{xy} \input xy \xyoption{all}
\usepackage[noend]{algorithmic}
\algsetup{linenodelimiter=.}

\newcommand{\F}{\mathbb{F}}
\newcommand{\C}{\mathbb{C}}

\newcommand{\N}{\mathbb{N}}

\newcommand{\mB}{\mathcal{B}}
\newcommand{\mV}{\mathcal{V}}
\newcommand{\mP}{\mathcal{P}}
\newcommand{\mQ}{\mathcal{Q}}
\newcommand{\mR}{\mathcal{R}}

\newcommand{\kpartial}[2]{\overrightarrow{^{#1}\partial{#2}}}
\newcommand{\canonicalvector}[1]{\mathbf{e}_{#1}}
\newcommand{\column}[1]{\mathbf{c}_{#1}}

\newcommand{\ot}{\,\overline{t}\,}
\newcommand{\odelta}{\,\overline{\delta}\,}
\newcommand{\ox}{\,\overline{x}\,}

\newcommand{\oa}{\,\overline{a}\,}

\newcommand{\orr}{\,\overline{r}\,}
\newcommand{\ow}{\,\overline{w}\,}
\newcommand{\oz}{\,\overline{z}\,}

\newcommand{\BP}{\mB_\mP}
\newcommand{\VP}{\mV_\mP}

\newcommand{\VT}{\mV_{\mathbb{T}}}

\DeclareMathOperator{\Ima}{Im}

\begin{document}

%
%
%
%
%
%
%
%
%

	\title[Transforming radical ODEs and PDEs into rational coefficients]
 {Transforming ODEs and PDEs with radical coefficients into rational coefficients}

\thanks{The authors are partially supported by FEDER/Ministerio de Ciencia, Innovaci\'on y Universidades - Agencia Estatal de Investigaci\'on/MTM2017-88796-P (Symbolic Computation: new challenges in Algebra and Geometry together with its applications).}
\thanks{The first, second and fourth authors are members of the Research Group ASYNACS (Ref. CT-CE2019/683). In addition, the second author is also partially supported by Direcci\'on General de Investigaci\'on e Innovaci\'on, Consejería de Educaci\'on e Investigaci\'on of the Comunidad de Madrid (Spain), and Universidad de Alcal\'a (UAH) under grant CM/JIN/2019-010.
}
\thanks{The third author is a member of the research group GADAC and is partially supported by Junta de Extremadura and FEDER funds (group FQM024).}

\author[Caravantes]{Jorge Caravantes}
\address{%
Dept. of Physics and Mathematics, University of Alcal\'a \\ E-28871 Alcal\'a de Henares (Madrid, Spain)}
\email{jorge.caravantes@uah.es}

\author[Sendra]{J. Rafael Sendra}
\address{%
	Dept. of Physics and Mathematics, University of Alcal\'a \\ E-28871 Alcal\'a de Henares (Madrid, Spain)}
\email{rafael.sendra@uah.es}

\author[Sevilla]{David Sevilla}
\address{%
	U. Center of M\'erida, University of Extremadura \\ Av. Santa Teresa de Jornet 38 \\ E-06800 M\'erida (Badajoz, Spain)}
\email{sevillad@unex.es}

\author[Villarino]{Carlos Villarino}
\address{%
	Dept. of Physics and Mathematics, University of Alcal\'a \\ E-28871 Alcal\'a de Henares (Madrid, Spain)}
\email{carlos.villarino@uah.es}




\begin{abstract}
We present an algorithm that transforms, if possible, a given ODE or PDE with radical function coefficients into one with rational coefficients by means of a rational change of variables. It also applies to systems of linear ODEs. It is based on previous work on reparametrization of radical algebraic varieties.
\end{abstract}

\maketitle

\section{Introduction}\label{sec:intro}

Within the vast variety of differential equations, certain subtypes are amenable to solving. A reasonable strategy is to transform given equations, by means of a change of variables, into subtypes for which algorithms exist. Let us consider an example, that will be revisited in Remark \ref{remark:detalles ejemplo intro}.

\begin{ex}\label{ex:intro}
	We consider the ODE
	\begin{equation}\label{eq:intro_radicales}
	\left(\,(14x+12)\sqrt{x}+(13x+4)\sqrt{x+1}\,\right)\,y+4(x^2+x)\, (y')^2=0.
	\end{equation}
	Maple 2018 expresses the nontrivial solution as the solution of an integral equation. 
	Using the techniques described in this article, we find the change
	\[
	x=r(z)=\frac{(z^2-1)^2}{4z^2}
	\]
	that transforms the equation into one with the unknown $Y(z)=y(r(z))$:
	\begin{equation}\label{eq:intro_algebraica}
	\left( {z}^{12}-2\,{z}^{10}-{z}^{8}+4\,{z}^{6}-{z}^{4}-2\,{z}^{2}+1
	\right) Y  +16\, {z}^{8}\, \left( Y'  \right) ^{2}=0.
	\end{equation}
	The nonzero solution is, as given by Maple 2018,
	\begin{equation}\label{eq:intro_solucion_racional}
	Y \left( z \right) =-{\frac { \left( {z}^{6}+3\,C\,{z}^{3}-3
			\,{z}^{4}+3\,{z}^{2}-1 \right) ^{2}}{576\,{z}^{6}}}.
	\end{equation}
	The inverse substitution, also obtained with our method, is $z=\sqrt{x}+\sqrt{x+1}$, so the general solution of the original equation is $y(x)=Y\left(\sqrt{x}+\sqrt{x+1}\,\right)$.
	
\end{ex}

A class of ODEs and PDEs for which there is a body of research is that of algebraic equations. They are polynomials in the unknown and its derivatives, and the coefficients of the polynomial are rational function in the variable(s) of the unknown. See \cite{MR3336231,MR3037915,MR2907990,CoxLittleOshea,ChauWinkler2010,MR3790007,MR3720262,MR3460289,10.1145/236869.237073}. In the previous example we transform a nonalgebraic equation into an algebraic one.

Our contribution is a method to find, if they exist, changes of variables that convert polynomial ODEs and PDEs with radical coefficients into equations that are algebraic. It is based on our previous work on rational reparametrization of radical varieties; see \cite{SS2011,SS2013} and specially \cite{SSV}.

The problem that we will study is the following. Given the variables $x_1,\ldots,x_n$, define a radical tower over $\C(x_1,\ldots,x_n)$,
\begin{equation}\label{eq:tower}
	\F_0=\C(x_1,\ldots,x_n)\subseteq\dots\subseteq\F_{m-1}\subseteq\F_{m}
\end{equation}
where $\F_i=\F_{i-1}(\delta_i)=\F_0(\delta_1,\ldots,\delta_i)$ with $\delta_i^{e_i}=\alpha_i\in\F_{i-1}$, $e_i\in\N$. Intuitively, $\F_m$ is the field of functions built using several radicals (possibly nested).

Now we consider, in the case of only one variable, an ODE
\begin{equation}\label{eq:la ecuacion}
	F(x,y,y^{\prime},\ldots,y^{s)})=0
\end{equation}
where $y$ is a function of $x$ and $F\in\F_{m}[\ow]$, where $\ow=(w_0,\ldots,w_s)$. That is, the equation is polynomial in $y$ and its derivatives, with the $x$ in the coefficients possibly appearing inside radicals. Our goal is to find, if it exists, a rational change of variable of the form
\[ x=r(z) \]
such that the new equation
\begin{equation}\label{eq:la ecuacion2}
	G(r(z),Y(z),\ldots, Y^{s)}(z))=0, \,\,\text{where $Y(z)=y(r(z))$},
\end{equation}
is an algebraic ODE in sense that the coefficients of $G$ are not just in $\F_m$ but are in fact rational functions.

In an analogous way, given a PDE that can be expressed as a polynomial in the partial derivatives of the unknown $y(x_1,\ldots,x_n)$, and whose coefficients are radical functions of the $x_i$, we compute, if it exists, a rational change of the $n$ variables such that the transformed equation has rational coefficients.

The structure of this article is as follows. The next section describes the algebraic preliminaries that will be used later. Section 3 describes the results on ordinary differential equations, Section 4 deals with linear systems of them, and Section 5 treats the case of partial differential equations.

\section{Preliminaries on radical varieties}\label{sec:radical varieties}

In this section we recall some results on radical varieties; for further details see \cite{SSV}. 
{Some of the notions and results in this section are framed within the field of algebraic geometry,
A good starting point for algebraic geometry from a computational point of view is \cite{CoxLittleOshea}.}

 We will express tuples of variables or functions by using bars over their names, like $\ox=(x_1,\ldots,x_n)$. 

Let $\F_m$ and $\delta_i$ {be} defined as in \eqref{eq:tower}; the latter can be regarded as functions of the $\ox$. A radical parametrization is a tuple $\mP=(\oa(\ox))$ of elements of $\F_m$ whose Jacobian has rank $n$. For a suitable election of branches we obtain a (usually nonrational) function with domain in $\C^n$ and whose image is Zariski dense in an $n$-dimensional variety; the Zariski closure of $\Ima(\mP)$ is the radical variety determined by $\mP$, denoted by $\VP$.

\begin{thm}[\cite{SSV},Theorem 3.11 (iii)]
	$\VP$ is an $n$-dimensional irreducible algebraic variety.
\end{thm}

The function $\mP$ can be factored as $\mR\circ\psi$, where $\psi$ is a nonrational function whose components are the $\odelta$ (the radicals in the construction) and $\mR$ is a rational function in the $\ox$ and the $\odelta$.

Another relevant algebraic variety in our construction is the tower variety of $\mP$, defined as the Zariski closure of $\Ima(\psi)$. It is also irreducible of dimension $n$. Therefore we have now a rational map $\mR\colon\VT\to\VP$. The tower variety contains useful information about $\mP$ and is the central object in this article's results.

One more ingredient is an incidence variety, the Zariski closure of the map $\ox\mapsto(\ox,\odelta,\oa)$, that we denote as $\BP$. This provides a projection $\pi\colon\BP\to\VP$. We define the tracing index of $\mP$ as the cardinal of the generic fibre of $\pi$. See \cite[Example 4.4]{SSV} for the calculation of the tracing index.

The following diagram summarizes the different elements introduced so far.
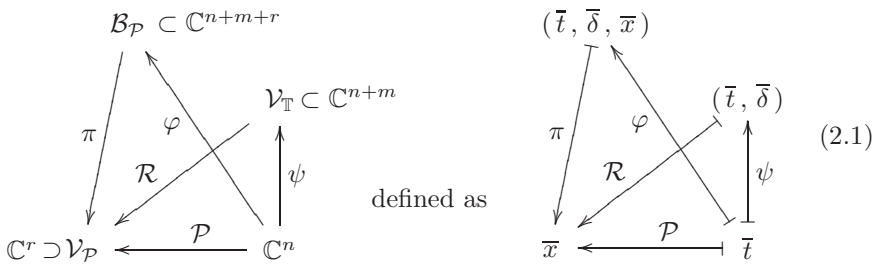
\begin{equation}\label{eq:diagrama2}
\hspace*{-3mm}
\xy
(0,0)*{\xy
	(-5,5)*++{\BP}="BP";
	"BP"+(12,0.5)*{\subset\C^{n+m+r}};
	(-11,-25)*++{\VP}="VP";
	"VP"+(-4,0)*{\supset}; "VP"+(-8,0)*++{\C^r};		
	(15,-25)*++{\C^n}="Cn";
	{\ar_{\displaystyle \pi} "BP"; "VP"};
	{\ar^(0.6){\displaystyle \varphi} "Cn"; "BP"};
	{\ar_(0.4){\displaystyle \mP} "Cn"; "VP"};
	(15,-5)*++{\VT}="VT";
	"VT"+(4,0)*{\subset}; "VT"+(11,0.5)*++{\C^{n+m}};
	{\ar_{\displaystyle \psi} "Cn"; "VT"};
	{\ar_(0.6){\displaystyle \mR} "VT"; "VP"};
	\endxy};
(30,-8)*{\mbox{defined as}};
(60,0)*{\xy
	(-5,5)*+{(\ot,\odelta,\ox)}="BP";
	(-11,-25)*++{\ox}="VP";
	(15,-25)*++{\ot}="Cn";
	{\ar@{|->}_{\displaystyle \pi} "BP"; "VP"};
	{\ar@{|->}^(0.6){\displaystyle \varphi} "Cn"; "BP"};
	{\ar@{|->}_(0.4){\displaystyle \mP} "Cn"; "VP"};
	(15,-5)*+{(\ot,\odelta)}="VT";
	{\ar@{|->}_{\displaystyle \psi} "Cn"; "VT"};
	{\ar@{|->}_(0.6){\displaystyle \mR} "VT"; "VP"};
	\endxy};
\endxy
\end{equation}

The case where the tracing index is 1 is of particular interest.

\begin{thm}[\cite{SSV}, Theorem 4.11]
	If the tracing index of $\mP$ is 1, then $\VP$ and $\VT$ are birationally equivalent. Therefore, $\VP$ is rational (parametrizable rationally in an invertible way) if and only if $\VT$ is rational.
\end{thm}

\section{ODEs}

In this section, we will use some of the notation introduced previously. In particular, $\F_m$ is the last field of a radical tower over $\C(x)$ and $\odelta=(\delta_1,\ldots,\delta_m)$ is the tuple of algebraic elements used in the construction of the tower. We also consider the ODE $F(x,y,y^{\prime},\ldots,y^{s)})=0$ introduced in \eqref{eq:la ecuacion}.

We start with a technical lemma that describes what happens to the unknown when a change of variable is applied.

\begin{lem}\label{lemma:transformacion}
	Let $ r(z)$ be a non-constant rational function, and let $Y(z)=y(r(z))$. Then
	\[
	\left( \begin{array}{c} y(r(z)) \\ y^{\prime}(r(z)) \\ \vdots \\ y^{s)}(r(z)) \end{array} \right) = L(z) \left( \begin{array}{c} Y(z) \\ Y^{\prime}(z) \\ \vdots \\ Y^{s)}(z) \end{array} \right),
	\]
	where $L$ is a lower triangular matrix with entries in $\C[r^\prime(z),\ldots,r^{s)}(z),\frac{1}{r^{\prime}(z)}]$, in particular of the form
	\[
	\frac{P(r^\prime(z),\ldots,r^{s)}(z))}{r^{\prime}(z)^{k}}
	\]
	with $k\in \N$ and $P$ a polynomial with complex coefficients. Moreover, $\det(L)$ is a power of $\frac{1}{r^{\prime}(z)}$.
\end{lem}
\begin{proof}
	We prove by induction that for $k\in \{0,\ldots,s\}$ it holds that
	\[
	r^{\prime}(z)^k y^{k)}(r(z))=Y^{k)}(z)+\sum_{j=1}^{k-1} R_{j}(r^{\prime},\ldots,r^{k)})Y^{j)}(z)
	\]
	for some $R_j\in \C[w_1,\ldots,w_k,\frac{1}{w_1}]$. For $k=0$ is obvious. Let us assume that it holds for $k=\ell$. Taking derivatives in the expression above for $k=\ell$ we have that, for some $Q_j\in \C[w_1,\ldots,w_{\ell+1},\frac{1}{w_1}]$,
	\[
	\ell r^{\prime}(z)^{\ell-1} r^{\prime\prime}(z) y^{\ell)}(r(z))+r^{\prime}(z)^{\ell+1} y^{\ell+1)}(r(z)) =
	\]
	\[
	= Y^{\ell+1)}(z)+\!\sum_{j=1}^{\ell}\! Q_j(r^{\prime},\ldots,r^{\ell+1)}) Y^{j)}(z).
	\]
	Applying the hypothesis of induction one ends the proof.
\end{proof}

\begin{rem}
	There exist transcendental functions $r(z)$ with radical or rational derivatives, like logarithms, arctangents and arcsines, but they fall outside the scope of this article, because we need $\odelta(r(z))$ to be rational, and this implies that $r(z)$ must be algebraic.
	
\end{rem}


\begin{thm}\label{th:equivalencia}
	Let $ r(z)$ be a non-constant rational function, and let $\oa(x)\in \F_{m}^{\ell}$ be the tuple of non-rational (function) coefficients of $F(x,\ow)$ w.r.t. $\ow$, taken in any order. The following are equivalent:
	\begin{enumerate}
		\item Equation (\ref{eq:la ecuacion2}) is algebraic for the substitution $x=r(z)$.
		\item All the components of $\oa(r(z))$ are rational.
	\end{enumerate}
\end{thm}

\begin{proof}
	The implication $2\Rightarrow1$ follows from 
	Lemma \ref{lemma:transformacion}. Since the matrix $L(z)$ in that lemma is invertible (it is lower triangular with nonzero determinant), the same argument can be applied to $L^{-1}$ to prove $1\Rightarrow2$.
\end{proof}

In order to use Theorem \ref{th:equivalencia}, we define the radical parametrization \linebreak$\mP=(x,\oa(x))$ and consider its associated radical variety $\VP$ as well as its tower variety $\VT$, see Section \ref{sec:radical varieties}.

\begin{cor}\label{cor:cambio}
	Suppose that $\VT$ be rational. Then, it holds that
	\begin{enumerate}
		\item $\VP$ is a rational curve.
		\item If $\mathcal{Q}(z)=(r(z),\odelta(r(z)))$ is a rational parametrization of $\VT$, then ODE (\ref{eq:la ecuacion2}), obtained by applying the change $x=r(z)$ in the equation
		(\ref{eq:la ecuacion}), is algebraic.
		\item Assume that $\mathcal{Q}(z)$ in the previous item is invertible and let $h(\oz)$ be its inverse. If $Y(z)$ is a solution of  (\ref{eq:la ecuacion2}), then $Y(h(x,\odelta(x)))$ is
		a solution of (\ref{eq:la ecuacion}).
	\end{enumerate}
\end{cor}

\begin{proof}\
	\begin{enumerate}
		\item Since $\VT$ is rational, by \cite[Theorem 4.9]{SSV}, $\VP$ is rational.
		\item Since the $\mR:\VT\to\VP$ in \eqref{eq:diagrama2} is rational, $\mR(\mQ(z))=(r(z),\oa(r(z))$ is a rational parametrization of $\VP$ and so $\oa(r(z))$ is algebraic with $r(z)\in \C(z)$. Now, the result follows from  Theorem \ref{th:equivalencia}.
		\item Let us consider from Section \ref{sec:intro} the equality $y(r(z))=Y(z)$. Then $Y(z)$ is a solution of \eqref{eq:la ecuacion2} when $y(x)$ is a solution of \eqref{eq:la ecuacion}. Since $z=h(\mQ(z))=h(r(z),\odelta(r(z)))$ we have
		\[
		y(r(z))=Y(h(r(z),\odelta(r(z)))).
		\]
		Since $r(z)$ is not constant (otherwise, $\mQ$ would also be constant), we can call $x=r(z)$. We get now $y(x)=Y(h(x,\odelta(x)))$.
	\end{enumerate}
\end{proof}

All previous ideas are algorithmically treated next. For this purpose, we will use two auxiliary algorithms (see \cite{SWP} for details):
\begin{itemize}
	\item \textsf{RatParamAlg} will  be an algorithm to decide whether an algebraic curve is rational and, in the affirmative case,
	to compute a proper (i.e. invertible) rational parametrization,
	\item \textsf{InvParamAlg} will be an algorithm to compute the inverse of a proper rational curve parametrization.
\end{itemize}

\begin{alg}[Transforming ODEs with radical coefficients	into rational coefficients]
	\begin{algorithmic}[1]
		$ $
		\renewcommand{\algorithmicrequire}{\textbf{Input:}\ }
		\renewcommand{\algorithmicensure}{\textbf{Output:}\ }

		\algorithmicrequire An ODE as in (\ref{eq:la ecuacion}).
		
		\algorithmicensure One of the following:
		\begin{itemize}
			\item a rational change of variable $x=r(z)$ such that, when applied to \eqref{eq:la ecuacion}, it yields an algebraic ODE as in \eqref{eq:la ecuacion2}; and a rational function $h(\oz)$ such that if $Y(z)$ is a solution of \eqref{eq:la ecuacion2}, then $Y(h(x,\odelta(x)))$ is a solution of \eqref{eq:la ecuacion}.
			\item ``No such rational reparametrization exists''
			\item ``No answer''
		\end{itemize}
		\STATE Collect in the tuple $\oa$ all nonrational coefficients of (\ref{eq:la ecuacion}).
		\STATE Compute the tower variety $\VT$ of $\mP=(x,\oa)$ (see \cite[Remark 4.7]{SSV}) and apply \textsf{RatParamAlg} to it.
		\IF {$\VT$ has a rational parametrization}
		\STATE Compute an invertible parametrization $\mathcal{Q}(z)=(r(z),\odelta(r(z)))$ and its inverse $h=\mathsf{InvParamAlg}(\mathcal{Q})$.
		\RETURN $x=r(z)$ and $h$.
		\ELSE
		\IF {The tracing index of $\mP$ is 1}
		\RETURN ``No such rational change of variables exists''
		\ELSE
		\RETURN ``No answer''
		\ENDIF
		\ENDIF
	\end{algorithmic}
\end{alg}

\begin{rem}
	If the tracing index of $\mP$ is 1, $\VP$ and $\VT$ are birationally equivalent; then, if $\VT$ is not rational, then $\VP$ cannot be rational. Therefore there exists no rational $r(z)$ that reparametrizes \eqref{eq:la ecuacion} into an algebraic ODE. See \cite[Theorem 4.11, (i)]{SSV}.
	
	
	On the other hand, when the tracing index is $T>1$, points in $\VP$ would generically have $T$ preimages by $\pi$ (see diagram \eqref{eq:diagrama2}), and they have the shape $(p,\odelta(p),\oa(p))$; the second block in that shape would consist of all the conjugates of $p$ by the $\odelta$, and we consider it possible that, by working on the Galois group of the extension, one can find more information on the radical parametrization, possibly even being able to find a simpler tower of radicals to convert to.
\end{rem}

\begin{rem}\label{remark:detalles ejemplo intro}
	Revisiting Example \ref{ex:intro} from the introduction, one can observe that the tower curve $\VT$ can be defined by $(x,\delta_1(x),\delta_2(x))=(x,\sqrt{x},\sqrt{x+1}\,)$. Then $\VT$ can be parametrized rationally by
	\[Q(z)=(\,r(z),\delta_1(r(z)),\delta_2(r(z))\,)=\left(\frac{(z^2-1)^2}{4z^2}, \frac{z^2-1}{2z}, \frac{z^2+1}{2z}\right),\]
	whose inverse is $h(A,B,C)=B+C$. The substitution $Y(z)=y(r(z))$ is the one we used to obtain \eqref{eq:intro_algebraica}. We found the solution \eqref{eq:intro_solucion_racional}, which means that $y(x)=Y(\, h(x,\delta_1(x),\delta_2(x))\,)=Y\left(\sqrt{x}+\sqrt{x+1}\,\right)$ solves \eqref{eq:intro_radicales}.
\end{rem}

\begin{ex}
	The following equation is based on \cite[Part I, Section C, 1.525]{Kamke}: for a parameter $a$, find $y(x)$ such that
	\[
	8(y')^3(x+1)^{3/2} - 2a(x+1)yy' + 2ay^2 = 0.
	\]
	Maple 2018 finds a solution after hundreds of seconds, and it is not explicit (it involves \texttt{RootOf} and integral signs). On the other hand, the coefficients have $\delta=\sqrt{x+1}$ so that we can consider $\VT$, the closure of the image of $x\mapsto(x,\sqrt{x+1})$. It is clearly parametrizable, for example as $\mQ(z)=(z^2-1,z)$, so the change $x=z^2-1$ should make all the coefficients rational; indeed, in terms of $Y(z)=y(z^2-1)$, we obtain
	\[
	(Y')^3 - azYY' + 2aY^2 = 0.
	\]
	Maple immediately finds the general solution (see also \cite[Example 7.1]{ChauWinkler2010})
	\[
	Y(z) = \frac{z^2}{4C} - \frac{z}{2aC^2} + \frac{1}{4a^2C^3}.
	\]
	To return to the variable $x$, we use the inverse of $\mQ(z)$, namely $h(A,B)=B$; the substitution is then $z=\delta(x)$, so
	\[
	y(x) = \frac{x+1}{4C} - \frac{\sqrt{x+1}}{2aC^2} + \frac{1}{4a^2C^3}
	\]
	is the general solution.
\end{ex}

\begin{ex}
	We now take the equation \cite[Part I, Section C, 6.166]{Kamke}:
	\[ayy'' +b \left(y'  \right) ^{2}-{\frac {y y' }{\sqrt {{c}^{2}+{x}^{2}}}}=0.\]
	Maple 2018 gives the solution in terms of the generalized hypergeometric function. We consider the tower curve $\VT$, defined by $(x,\delta(x))=\left(x,\sqrt{c^2+x^2}\right)$. Then, $\VT$ (the hyperbola $\delta^2-x^2=c^2$) can be rationally parametrized by
	\[
	Q(z)=\left(r(z),\delta(r(z))\right)=\left({\frac {-{c}^{2}+{z}^{2}}{2z}}, {\frac {{c}^{2}+{z}^{2}}{2z}}\right).
	\]
	We now perform the change $x=r(z)$ and  we get, for $Y(z)=y(r(z))$, the equation
	\[(  a{c}^{2}z+  a{z}^{3})YY''+
	(b{c}^{2}z+b{z}^{3})(Y') ^{2}+
	(2 a{c}^{2}- {c}^{2}- {z}^{2})YY' =0,\]
	whose solution $Y(z)$, given by Maple 2018, is
	\[
	\left(
	\frac{\left( \left( -\sqrt [a]{z}{a}^{2}{c}^{2}-\sqrt [a]{z}a{c}^{2}+{z}^{{\frac {
					1+2\,a}{a}}}{a}^{2}-{z}^{{\frac {1+2\,a}{a}}}a \right)  C_1+
		\left( {a}^{2}z-z \right) C_2\right)
	}
	{az(1+a)(a-1)\,/\,(b+a)}
	\right) ^{\displaystyle\frac{a}{b+a}}.
	\]
	The inverse of $Q$ is defined by $h(A,B)=A+B$, so we substitute\linebreak $z=x+\sqrt{c^2+x^2}$ to get $y(x)=Y(h(x,\delta(x)))=Y\left(x+\sqrt{c^2+x^2}\,\right)$, that solves the original equation.
\end{ex}

\section{Systems of linear ODEs}

In this section we extend the results of the previous section to the case of systems of linear ODEs.

\begin{thm}
	Consider a system of linear ODEs with unknowns\linebreak $y_1(x),\ldots,y_n(x)$ and coefficients $a^{(k)}_{i,j}(x)\in\F_m$:
	\[\renewcommand{\arraystretch}{1.5}
	\begin{array}{c}
	a_{1,0}^{(1)}\cdot y_1 + a_{1,1}^{(1)}\cdot {y_1}' + a_{1,2}^{(1)}\cdot {y_1}'' + \cdots + a_{2,0}^{(1)}\cdot y_2 + \cdots + a_{n,0}^{(1)}\cdot y_n +\cdots = b^{(1)} \\
	a_{1,0}^{(2)}\cdot y_1 + a_{1,1}^{(2)}\cdot {y_1}' + a_{1,2}^{(2)}\cdot {y_1}'' + \cdots + a_{2,0}^{(2)}\cdot y_2 + \cdots + a_{n,0}^{(2)}\cdot y_n +\cdots = b^{(2)} \\
	\vdots \\
	a_{1,0}^{(l)}\cdot y_1 + a_{1,1}^{(l)}\cdot {y_1}' + a_{1,2}^{(l)}\cdot {y_1}'' + \cdots + a_{2,0}^{(l)}\cdot y_2 + \cdots + a_{n,0}^{(l)}\cdot y_n +\cdots = b^{(l)} \\
	\end{array}
	\]
	(the superindex denotes the equation, and the subindexes denote the unknown and its derivation order). Then, for every $r(z)\in\C(z)$ the following are equivalent:
	\begin{enumerate}
		\item All of the $a_{i,j}^{(k)}(r(z))$ are rational.
		\item The system resulting from the reparametrization $x=r(z)$ has rational coefficients.
	\end{enumerate}
\end{thm}

\begin{proof}
	It is a consequence of applying Lemma \ref{lemma:transformacion} to each variable.
\end{proof}

\begin{rem}
	The theorem also holds for slightly more general systems of ODEs: those where no monomial contains a product of derivatives of $y_i$ and $y_j$ for $i\neq j$. We do now know if the result still holds for systems of polynomial ODEs.
\end{rem}

The results of the previous section can be adapted. Namely, by defining a radical parametrization $\mP=(x,\oa)$ where $\oa$ is a tuple of all the non-rational coefficients in any order, we can work with $\VP$ and $\VT$ in the same way.

\begin{cor}
	Suppose that $\VT$ is rational and let $\mathcal{Q}(z)=(r(z),\odelta(r(z)))$ be an invertible rational parametrization of $\VT$ with inverse $h(\oz)$. Then the linear system obtained by applying the change $x=r(z)$ to the original one is algebraic. Also, if $Y_1(z),\ldots,Y_n(z)$ is a solution of the new system, then \linebreak$Y_1(h(x,\odelta(x))),\ldots,Y_n(h(x,\odelta(x)))$ is
	a solution of the original system.
\end{cor}

\section{PDEs}

In this section we present similar results for the case of PDEs. We will denote partial derivatives with subindexes: $\displaystyle y_{i_1\ldots i_s}(x_1,\ldots,x_n) = \frac{\partial^s y}{\partial x_{i_1}\cdots\partial x_{i_s}}$. We will consider a PDE
\begin{equation}\label{eq:PDE original}
F(x_1,\ldots,x_n,y,y_1(\ox),\ldots,y_n(\ox),y_{11}(\ox),\ldots)=0
\end{equation}
where $F$ is a polynomial with coefficients in a radical tower $\F_m$ over $\C(\ox)$, and the problem of converting it to a PDE with a change of variables $x_i=r_i(z_1,\ldots,z_n)$
\begin{equation}\label{eq:PDE cambiada}
G(r_1(\oz),\ldots,r_n(\oz),Y(\oz),Y_1(\oz),\ldots,Y_n(\oz),Y_{11}(\oz),\ldots)=0
\end{equation}
which is algebraic (that is, the coefficients of $G$ are rational).

\begin{defn}
	Let $k$ be a positive integer and let $f$ be continuously differentiable of order $k$. Its \emph{generalized gradient vector up to order $k$} is the column vector
	\[
	\kpartial{k}{f} = (f,f_1,\ldots,f_n,\ f_{11},f_{12},\ldots,f_{1n},f_{22},\ldots,f_{nn},\ \ldots,\ldots,\ f_{n\cdots n})^T
	\]
	so that:
	\begin{itemize}
		\item it contains the derivatives of order 0, then order 1, etc. up to order $k$, sorted lexicographically within the same derivation order;
		\item when equality of mixed partials applies, only the first one (lexicographically) is included. For example, since $f_{12}=f_{21}$, we include the first one but not the second one.
	\end{itemize}
\end{defn}

\begin{lem}\label{lemma:transformacion PDE}
	Let $r_i\in\C(\oz)$, $i=1,\ldots,n$ be a change of variables, i.e. whose Jacobian is invertible. For any $y(x_1,\ldots,x_n)$ define $Y(\oz)=y(r_1(\oz),\ldots,r_n(\oz))$. Then
	\[
	\kpartial{k}{Y}(\oz) = M(\oz) \cdot \kpartial{k}{y}(\orr(\oz))
	\]
	where $M(\oz)$ is an invertible matrix whose entries are polynomials in the partial derivatives of the $r_i$.
\end{lem}
\begin{proof}
	
	Each derivative of $Y$ is computed, by repeated application of the chain rule, as a linear combination of the corresponding derivatives of $y(\orr)$. This proves the claim on the elements of $M$, and only the invertibility of $M$ remains to be proven, which we do next.
	
	First, we establish a sort of reciprocal formula. Note that
	\[
	\kpartial{1}{y}(\orr(\oz)) = \left(\begin{array}{c|c}
	1 & 0 \ldots 0 \\
	\hline
	\begin{array}{c}
	0 \\ \vdots \\ 0 \\
	\end{array} & (\mathrm{J}_{\orr})^T
	\end{array}\right)^{-1} \cdot\  \kpartial{1}{Y}(\oz),
	\]
	which this is possible because the Jacobian matrix $\mathrm{J}_{\orr}$ is invertible. Reasoning as above, we obtain linear relations
	\[
	\kpartial{k}{y}(\orr(\oz)) = N(\oz) \cdot \kpartial{k}{Y}(\oz),
	\]
	where the entries of $N$ are rational functions in the partial derivatives of $\orr$.
	
	We will finish the proof by showing that $M\cdot N=Id$. From the two equations,
	\begin{equation}\label{eq:MN=Id}
	\kpartial{k}{Y} = (M\cdot N)\ \kpartial{k}{Y}.
	\end{equation}
	Let us write $\canonicalvector{i}$ for a column vector whose elements are all zero except for a one at the $i$-th position (i.e. the $i$-th vector of the canonical basis).
	
	Note that, by the inverse function theorem applied generically to $\orr$, for any continuously differentiable $g$ there exists $\tilde{g}$ such that $\tilde{g}(\orr)=g$. This allows us to prescribe $Y(\oz)$ so that it will be equal to $y(\orr)$ for some $y$.
	
	Therefore, it makes sense to consider a monomial $Y=z_1^{u_1}\cdots z_n^{u_n}$. Since its degree is $\sum u_i$:
	\begin{itemize}
		\item its derivatives of order $>\sum u_i$ are zero;
		\item its derivatives of order $\sum u_i$ are zero except one, which is a nonzero constant.
	\end{itemize}
	
	Now we will do an induction on the order $h$ of derivation in order to show that the $i$-th column of $M\cdot N$, which we will denote by $\column{i}$, is equal to $\canonicalvector{i}$. We will make use of the fact that
	\[
	M\cdot N\cdot\left(\hspace{-1ex}\begin{array}{c} \lambda_1 \\ \lambda_2 \\ \vdots \\ \end{array}\hspace{-1ex}\right) = \lambda_1\column{1}+\lambda_2\column{2}+\cdots
	\]
	\begin{itemize}
		\item Case $h=0$: let $Y=1$. Then $\kpartial{k}{Y} = \canonicalvector{1}$, and from \eqref{eq:MN=Id} we deduce that $\column{1}=\canonicalvector{1}$.
		\item Case $h=1$: let $Y=z_i$. Then $\kpartial{k}{Y} = z_i\cdot\canonicalvector{1}+\canonicalvector{i+1}$. From \eqref{eq:MN=Id} we have that $z_i\cdot\canonicalvector{1}+\canonicalvector{i+1} = z_i\cdot\column{1}+\column{i+1}$. Since $\column{1}=\canonicalvector{1}$ by the previous case, then $\column{i+1}=\canonicalvector{i+1}$.
		\item Assume that the columns corresponding to orders of derivation $<h$ (let us say that there are $q$ of them) have been proven to be $\canonicalvector{1},\ldots,\canonicalvector{q}$. We proceed as in the previous case. Let $Y$ be a monomial of degree $h$. Then, by a previous observation, $\kpartial{k}{Y} = w_1\canonicalvector{1}+\cdots+w_q\canonicalvector{q}+c\canonicalvector{q'}$, where the $w_i$ are continuously differentiable (they are lower-order derivatives of $Y$), $c$ is a nonzero constant, and $q'>q$ is the position in the column vector corresponding to the only derivative of $Y$ of order $h$ that is nonzero (because for any index greater than $q$ the corresponding entry in the gradient is zero, except for position $q'$). Once more, by \eqref{eq:MN=Id},
		\[
		w_1\canonicalvector{1}+\cdots+w_q\canonicalvector{q}+c\canonicalvector{q'}=w_1\column{1}+\cdots+w_q\column{q}+c\column{q'}.
		\]
		By the induction hypothesis, $\column{i}=\canonicalvector{i}$ for $i\leq q$. Therefore, $\column{q'}=\canonicalvector{q'}$. Finally, for every column corresponding to a derivative of order $h$ there is such a monomial, and the induction is concluded; it follows that $M$ is invertible (and its inverse is precisely $N$).
	\end{itemize}
\end{proof}

Now we follow a path analogous to that of ODEs.

\begin{thm}\label{th:equivalencia PDE}
	With the previous notations, the following statements are equivalent:
	\begin{enumerate}
		\item The PDE \eqref{eq:PDE cambiada} is algebraic.
		\item All the components of the vector $\oa(\orr(\oz))$ are rational, where $\oa(\ox)$ are the non-rational coefficients of the PDE \eqref{eq:PDE original}.
	\end{enumerate}
\end{thm}
\begin{proof}
	As in Lemma \ref{lemma:transformacion}, both implications are proven by using the matrix $M(\oz)$ and its inverse from Lemma \ref{lemma:transformacion PDE}.
\end{proof}

\begin{cor}
	Suppose that $\VT$ be unirational (i.e. rationally parametrizable). Then, it holds that
	\begin{enumerate}
		\item $\VP$ is unirational.
		\item If $\mathcal{Q}(\oz)=(\orr(\oz),\odelta(\orr(\oz)))$ is a rational parametrization of $\VT$, then \eqref{eq:PDE cambiada}, obtained by applying the change $\ox=\orr(\oz)$ in \eqref{eq:PDE original}, is algebraic.
		\item Assume that $\mathcal{Q}(\oz)$ in the previous item is invertible and let $\overline{h}(\oz)$ be its inverse. If $Y(\oz)$ is a solution of \eqref{eq:PDE cambiada}, then $Y(\overline{h}(\ox,\odelta(\ox)))$ is a solution of \eqref{eq:PDE original}.
	\end{enumerate}
\end{cor}
\begin{proof} 1 and 2 are consequences of \cite[Theorem 4.11, (ii)]{SSV} and Theorem \ref{th:equivalencia PDE}. On the other hand, 3 is analogous to Corollary \ref{cor:cambio}, 3.
\end{proof}

In order to apply the reparametrization algorithm to the case of PDEs, a parametrization algorithm for $n$-dimensional algebraic varieties is required. For $n=2$, surface parametrization algorithms do exist, see for example \cite{Schicho}; no general constructive results are known to us.

%

\begin{ex}
	We consider an equation based on \cite[Section II.B.88, Example 2]{Zwillinger}: find $y(x_1,x_2,x_3)$ such that
	\[
		(-\sqrt{x_2}\,y_3+2\,y_1)\sqrt{x_1+\sqrt{x_2}} + 2\sqrt{x_2}\,y_2 - y^2 - y_1.
	\]
	Maple 2018 fails to solve it.
	
	In this case the radicals are nested, but it does not affect our construction. The radical tower is generated by $\delta_1$ such that $\delta_1^2=x_2$, and $\delta_2$ such that $\delta_2^2=x_1+\delta_1$. The tower variety is parametrized (nonrationally) by $(x_1,x_2,x_3,\delta_1,\delta_2)$ and it does admit a rational parametrization
	\[
	\mQ(z_1,z_2,z_3)=(z_1,(z_2^2-z_1)^2,z_3,z_2^2-z_1,z_2)
	\]
	that can be deduced from the defining equations of the $\odelta$. If we substitute the $\ox$ by the first three components of $\mQ$ we obtain the new equation in $Y(z_1,z_2,z_3)$
	\[
	(2z_2-1)\,Y_1 + z_2(z_1-z_2^2)\,Y_3 - Y^2 + Y_2.
	\]
	With Maple we obtain the following general solution:
	\[
	Y = \frac2{1-2z_2+F(w_1,w_2)},\quad\mbox{where}\quad w_1 = z_2^2-z_1-z_2,
	\]
	\[
	w_2 = \frac56 z_2^3-\frac54z_2^2+\frac34z_2-\frac16+\frac12z_1z_2^2-\frac12z_1^2-z_1z_2+\frac34z_1+z_3.
	\]
	The inverse of $\mQ$ is $h(A,B,C,D,E)=(A,E,C)$. Therefore the inverse of the change of variables is $(x_1,x_2,x_3)\mapsto(x_1,\delta_2,x_3)$ and the general solution of the original equation is, after simplification,
	\[
	y = \frac2{1-2\sqrt{x_1+\sqrt{x_2}}+F(q_1,q_2)},\quad\mbox{where}\quad q_1 = \sqrt{x_2}-\sqrt{x_1+\sqrt{x_2}},
	\]
	\[
	q_2 = \frac{(10\sqrt{x_2}-2x_1+9)\sqrt{x_1+\sqrt{x_2}}+(6x_1-15)\sqrt{x_2}}{12} - \frac{x_1}{2} + x_3 - \frac16.
	\]
\end{ex}

\begin{ex}
	We take the equation \cite[Part II, Section E, 6.62]{Kamke}: find the $u(x_1,x_2)$ such that
	\[
	u_1^2+u_2^2=\frac{1}{\sqrt{x_1^2+x_2^2}}.
	\]
	Maple 2018 does not give any solution.
	
	The surface $\VT$ is the image of $(x_1,x_2,\sqrt{x_1^2+x_2^2})$ which is a cone, and can be parametrized rationally as
	\[
	\mQ(z_1,z_2) = \left( \frac{2z_1z_2}{z_1^2+1},\frac{z_2(z_1^2-1)}{z_1^2+1},z_2 \right).
	\]
	Substituting the first two components of this into the $\ox$ we have the new equation
	\[
	Y_1^2 (z_1^2+1)^2 + 4 z_2^2 Y_2^2 - 4z_2 = 0.
	\]
	It is clearly an equation with separated variables, and so Maple provides the general solution
	\[
	Y(z_1,z_2) = F_1(z_1)+F_2(z_2) \quad\mbox{with}\quad F_1'(z_1)^2=\frac{K}{(z_1^2+1)^2}, \ F_2'(z_2)^2=\frac{4z_2-K}{4z_2^2}
	\]
	which is, explicitly,
	\[
	Y(z_1,z_2) = \sqrt{K}\arctan(z_1) - \sqrt{K}\arctan\left(\frac{\sqrt{4z_2-K}}{\sqrt{K}}\right) + \sqrt{4z_2-K}.
	\]
	In order to recover solutions of the original equation we consider the inverse of $\mQ$, which is $h(A,B,C)=(A/(C-B),C)$. Then the change that recovers the solutions of the original equation is
	\[
	(z_1,z_2)=\left(\frac{x_1}{\sqrt{x_1^2+x_2^2}-x_2},\sqrt{x_1^2+x_2^2}\right).
	\]
\end{ex}




\end{document}